\documentclass[10pt,reqno]{amsart}
\usepackage[hypertex]{hyperref}

\usepackage{longtable}
\usepackage{amsmath,amsthm}
\usepackage{amssymb}
\usepackage{euscript}

\numberwithin{equation}{subsection}

\newcommand{\sqsp}{\renewcommand{\baselinestretch}{1.1}\tiny\normalsize}

\newcommand{\relphantom[1]}{\mathrel{\phantom{#1}}}
\newtheorem{theorem}[subsection]{Theorem}
\newtheorem{lemma}[subsection]{Lemma}
\newtheorem{proposition}[subsection]{Proposition}
\newtheorem{corollary}[subsection]{Corollary}

\theoremstyle{definition}
\newtheorem{definition}[subsection]{Definition}
\newtheorem{example}[subsection]{Example}

\newcommand{\Abar}{\overline{A}}
\newcommand{\Atilde}{\widetilde{A}}

\newcommand{\bk}{\mathbf{k}}
\newcommand{\bC}{\mathbf{C}}
\newcommand{\bN}{\mathbf{N}}
\newcommand{\bZ}{\mathbf{Z}}

\newcommand{\xtwo}{x_1 \otimes x_2}
\newcommand{\ytwo}{y_1 \otimes y_2}
\newcommand{\ztwo}{z_1 \otimes z_2}

\newcommand{\muop}{\mu^{op}}

\newcommand{\dotalpha}{\cdot_\alpha}
\newcommand{\diamondalpha}{\diamond_\alpha}

\newcommand{\andspace}{\quad\text{and}\quad}

\newcommand{\byhomass}{\quad\text{(by Hom-associativity)}}
\newcommand{\bycomm}{\quad\text{(by commutativity)}}
\newcommand{\bymult}{\quad\text{(by multiplicativity)}}

\begin{document}

\title{A twisted generalization of Novikov-Poisson algebras}
\author{Donald Yau}

\begin{abstract}
Hom-Novikov-Poisson algebras, which are twisted generalizations of Novikov-Poisson algebras, are studied.  Hom-Novikov-Poisson algebras are shown to be closed under tensor products and several kinds of twistings.  Necessary and sufficient conditions are given under which Hom-Novikov-Poisson algebras give rise to Hom-Poisson algebras.
\end{abstract}



\address{Department of Mathematics\\
    The Ohio State University at Newark\\
    1179 University Drive\\
    Newark, OH 43055, USA}
\email{dyau@math.ohio-state.edu}

\date{\today}
\maketitle

\sqsp

%%%%%%%%%%%%%%%%%%%%%%
\section{Introduction}
%%%%%%%%%%%%%%%%%%%%%%

A Novikov algebra has a binary operation such that the associator is left-symmetric and that the right multiplication operators commute.  Novikov algebras play a major role in the studies of Hamiltonian operators and Poisson brackets of hydrodynamic type \cite{bn,dn1,dn2,dg1,dg2,dg3}.  The left-symmetry of the associator implies that every Novikov algebra is Lie admissible, i.e., the commutator bracket $[x,y] = xy - yx$ gives it a Lie algebra structure.

% my hom-novikov algebras, hou-bai paper

In \cite{yau5} the author initiated the study of a twisted generalization of Novikov algebras, called Hom-Novikov algebras.  A Hom-Novikov algebra $A$ has a binary operation $\mu$ and a linear self-map $\alpha$, and it satisfies some $\alpha$-twisted versions of the defining identities of a Novikov algebra.  In \cite{yau5} several constructions of Hom-Novikov algebras were given and some low dimensional Hom-Novikov algebras were classified.  Using some of the definitions and results in \cite{yau5}, a corresponding generalization of Poisson brackets of hydrodynamic type was studied in \cite{hb}.  Other Hom-type algebraic structures are studied in \cite{ms,ms2,ms3} and the author's papers listed in the references.

% Xu's Novikov-Poisson tensor theory

Novikov algebras, like Lie algebras, are not closed under tensor products in a non-trivial way.  In order to have a satisfactory tensor theory of Novikov algebras, certain extra structures are needed.  In the case of Lie algebras, the relevant structure for a tensor theory is a Poisson algebra structure.  A Poisson algebra has simultaneously a Lie algebra structure and a commutative associative algebra structure, satisfying the Leibniz identity.  Using Poisson algebra as a motivation, Xu in \cite{xu1} defined a Novikov-Poisson algebra as a Novikov algebra that is also equipped with a commutative associative product, satisfying some compatibility conditions.  Novikov-Poisson algebras are closed under tensor products \cite{xu1} and some perturbations of the structure maps \cite{xu2}.  The relationship between Novikov-Poisson algebras and Hamiltonian super-operators was discussed in \cite{xu3}.

% purpose of this paper

The purpose of this paper is to study Hom-Novikov-Poisson algebras, which generalize Novikov-Poisson algebras in the same way that Hom-Novikov algebras generalize Novikov algebras.  In section \ref{sec:hnp} we defined Hom-Novikov-Poisson algebras and discuss some of their basic properties.  It is shown that Hom-Novikov-Poisson algebras are closed under twisting by weak morphisms and that they arise from Novikov-Poisson algebras.  Several examples of Hom-Novikov-Poisson algebras are given.

In section \ref{sec:tensor}, it is shown that Hom-Novikov-Poisson algebras are closed under tensor products in a non-trivial way, generalizing a result in \cite{xu1}.  This tensor product is shown to be compatible with the twisting constructions in section \ref{sec:hnp}.

In section \ref{sec:perturb}, it is shown that every multiplicative Hom-Novikov-Poisson algebra can be perturbed in several ways by its own twisting map and suitable elements.  These results reduce to some of those in \cite{xu2} when the twisting map is the identity map.

If a Hom-Novikov-Poisson algebra gives rise to a Hom-Poisson algebra \cite{ms3,yau10} via the commutator bracket of the Hom-Novikov product, then it is called admissible.  In section \ref{sec:poisson}, a necessary and sufficient condition for admissibility is given, which generalizes an observation in \cite{zbm}.  It is then shown that admissibility is preserved by the twisting constructions in section \ref{sec:hnp}, the tensor products in section \ref{sec:tensor}, and the perturbations in section \ref{sec:perturb}.

%%%%%%%%%%%%%%%%%%%%%%%%%%%%%%%%%%%%%%%
\section{Hom-Novikov-Poisson algebras}
\label{sec:hnp}
%%%%%%%%%%%%%%%%%%%%%%%%%%%%%%%%%%%%%%%

The purposes of this section are to introduce Hom-Novikov-Poisson algebras and to discuss some basic properties and examples of these objects.  Before we give the definition of a Hom-Novikov-Poisson algebra, let us first fix some notations.

%%%%%%%%%%%%%%%%%%%%%%%%
\subsection{Notations}

We work over a fixed field $\bk$ of characteristic $0$.  For a linear self-map $\alpha \colon V \to V$, denote by $\alpha^n$ the $n$-fold composition of $n$ copies of $\alpha$, with $\alpha^0 \equiv Id$.  If $\mu \colon V^{\otimes 2} \to V$ is a linear map, we often abbreviate $\mu(x,y)$ to $xy$ for $x,y \in V$.  Denote by $\muop \colon V^{\otimes 2} \to V$ the opposite map, i.e., $\muop = \mu\tau$, where $\tau \colon V^{\otimes 2} \to V^{\otimes 2}$ interchanges the two variables.

%%%%%%%%%%%%%%%%%%%%%
\begin{definition}
\label{def:hommodule}
\begin{enumerate}
\item
A \textbf{Hom-module} is a pair $(A,\alpha)$ in which $A$ is a $\bk$-module and $\alpha \colon A \to A$ is a linear self-map, called the twisting map.  A \textbf{morphism} $f \colon (A,\alpha_A) \to (B,\alpha_B)$ of Hom-modules is a linear map $f \colon A \to B$ such that $f\alpha_A = \alpha_Bf$.
\item
A \textbf{Hom-algebra} is a triple $(A,\mu,\alpha)$ in which $(A,\alpha)$ is a Hom-module and $\mu \colon A^{\otimes 2} \to A$ is a bilinear map.  Such a Hom-algebra is \textbf{commutative} if $\mu = \muop$.  It is \textbf{multiplicative} if $\alpha\mu = \mu \alpha^{\otimes 2}$.
\item
A \textbf{double Hom-algebra} is a quadruple $(A,\mu_1,\mu_2,\alpha)$ in which $(A,\alpha)$ is a Hom-module and each $\mu_i \colon A^{\otimes 2} \to A$ is a bilinear map.
\item
A double Hom-algebra $(A,\mu_1,\mu_2,\alpha)$ is \textbf{multiplicative} if $\alpha\mu_i = \mu_i\alpha^{\otimes 2}$ for $i=1,2$.
\item
A \textbf{weak morphism} $f \colon A \to B$ of double Hom-algebras is a linear map such that $f\mu_i = \mu_i f^{\otimes 2}$ for $i=1,2$.  A \textbf{morphism} $f \colon A \to B$ of double Hom-algebras is a weak morphism such that $f\alpha_A = \alpha_Bf$.
\end{enumerate}
\end{definition}
%%%%%%%%%%%%%%%%%%%%%

%%%%%%%%%%%%%%%%%%%%%
\begin{definition}
\label{def:homassociator}
Let $(A,\mu,\alpha)$ be a Hom-algebra.  Its \textbf{Hom-associator} \cite{ms} $as_A \colon A^{\otimes 3} \to A$ is defined as
\[
as_A = \mu(\mu \otimes \alpha - \alpha \otimes \mu),
\]
i.e.,
\[
as_A(x,y,z) = (xy)\alpha(z) - \alpha(x)(yz)
\]
for $x,y,z \in A$.
\end{definition}
%%%%%%%%%%%%%%%%%%%%%

%%%%%%%%%%%%%%%%%%%%%
\begin{definition}
\label{def2:homassociator}
Let $(A,\mu_1,\mu_2,\alpha)$ be a double Hom-algebra.  Its \textbf{mixed Hom-associator} $as_A \colon A^{\otimes 3} \to A$ is defined as
\[
as_A = \mu_1(\mu_2 \otimes \alpha) - \mu_2(\alpha \otimes \mu_1).
\]
The Hom-associator with respect to $\mu_i$ is denoted by $as_{\mu_i}$, i.e.,
\[
as_{\mu_i} = \mu_i(\mu_i \otimes \alpha - \alpha \otimes \mu_i)
\]
for $i=1,2$.
\end{definition}
%%%%%%%%%%%%%%%%%%%%%

Let us now recall Hom-associative algebras from \cite{ms}; see also \cite{yau,yau2}.

%%%%%%%%%%%%%%%%%%%%%
\begin{definition}
\label{def:homass}
A \textbf{Hom-associative algebra} is a Hom-algebra $(A,\mu,\alpha)$ such that $as_A = 0$, i.e.,
\begin{equation}
\label{homassociative}
(xy)\alpha(z) = \alpha(x)(yz)
\end{equation}
for all $x,y,z \in A$.
\end{definition}
%%%%%%%%%%%%%%%%%%%%%

The condition $as_A = 0$ is called \textbf{Hom-associativity}. An associative algebra is a multiplicative Hom-associative algebra with $\alpha = Id$.

Next we recall the definition of a Hom-Novikov algebra from \cite{yau5}.

%%%%%%%%%%%%%%%%%%%%%
\begin{definition}
\label{def:homnovikov}
A \textbf{Hom-Novikov algebra} is a Hom-algebra $(A,\mu,\alpha)$ such that
\begin{subequations}
\label{homnovikov}
\begin{align}
as_A(x,y,z) &= as_A(y,x,z),\label{leftsymmetric}\\
(xy)\alpha(z) &= (xz)\alpha(y)\label{rightmult}
\end{align}
\end{subequations}
for all $x,y,z \in A$.
\end{definition}
%%%%%%%%%%%%%%%%%%%%%

A \textbf{Novikov algebra} is a multiplicative Hom-Novikov algebra with $\alpha = Id$. The condition \eqref{leftsymmetric} means that the Hom-associator is left-symmetric, i.e., symmetric in the first two variables.  The condition \eqref{rightmult} means that
\[
R_yR_{\alpha(z)} = R_zR_{\alpha(y)}
\]
for all $y,z \in A$, where $R_y$ denotes right multiplication by $y$.

We can now define Hom-Novikov-Poisson algebras.

%%%%%%%%%%%%%%%%%%%%%
\begin{definition}
\label{def:hnp}
A \textbf{Hom-Novikov-Poisson algebra} is a double Hom-algebra $(A, \cdot, \ast, \alpha)$ such that
\begin{enumerate}
\item
$(A,\cdot,\alpha)$ is a commutative Hom-associative algebra,
\item
$(A,\ast,\alpha)$ is a Hom-Novikov algebra,
\end{enumerate}
and that the following two compatibility conditions
\begin{subequations}
\label{hnp}
\begin{align}
as_A(x,y,z) &= as_A(y,x,z),\label{mixedass}\\
(x \cdot y) \ast \alpha(z) &= (x \ast z) \cdot \alpha(y)\label{rightmult2}
\end{align}
\end{subequations}
hold for all $x,y,z \in A$, where $as_A$ is the mixed Hom-associator in Definition \ref{def2:homassociator}.
\end{definition}
%%%%%%%%%%%%%%%%%%%%%

A \textbf{Novikov-Poisson algebra} \cite{xu1,xu2} is a multiplicative Hom-Novikov-Poisson algebra with $\alpha = Id$.  Notice the similarity between \eqref{hnp} and \eqref{homnovikov}.  Indeed, \eqref{mixedass} means that the mixed Hom-associator $as_A$ is left-symmetric.  Expanding the mixed Hom-associator (Definition \ref{def2:homassociator}) in terms of $\cdot$, $\ast$, and $\alpha$, we can rewrite the condition \eqref{mixedass} as
\begin{equation}
\label{mixed}
(x \ast y) \cdot \alpha(z) - \alpha(x) \ast (y \cdot z) =  (y \ast x) \cdot \alpha(z) - \alpha(y) \ast (x \cdot z).
\end{equation}
Likewise, \eqref{rightmult2} means that
\[
R^\cdot_y R^\ast_{\alpha(z)} = R^\ast_z R^\cdot_{\alpha(y)}
\]
for all $y, z \in A$, where $R^\cdot_y$ (resp., $R^\ast_z$) is right multiplication by $y$ (resp., $z$) using $\cdot$ (resp., $\ast$).

The following observation says that there is another way to state the compatibility condition \eqref{rightmult2} in a Hom-Novikov-Poisson algebra.  It will be used many times below.

%%%%%%%%%%%%%%%%%%%%%
\begin{lemma}
\label{lem:rightmult}
Let $(A, \cdot, \ast, \alpha)$ be a double Hom-algebra in which $\cdot$ is commutative.  Then
\begin{equation}
\label{rightmult'}
(x \cdot y) \ast \alpha(z) = (x \ast z) \cdot \alpha(y)
\end{equation}
for all $x,y,z \in A$ if and only if
\begin{equation}
\label{rightmult''}
(x \cdot y) \ast \alpha(z) = \alpha(x) \cdot (y \ast z)
\end{equation}
for all $x,y,z \in A$.  In particular, if $A$ is a Hom-Novikov-Poisson algebra, then
\[
(x \cdot y) \ast \alpha(z) = (x \ast z) \cdot \alpha(y) = \alpha(x) \cdot (y \ast z)
\]
for all $x,y,z \in A$.
\end{lemma}
%%%%%%%%%%%%%%%%%%%%%

\begin{proof}
By the commutativity of $\cdot$ we have
\[
(x \cdot y) \ast \alpha(z)
= (y \cdot x) \ast \alpha(z)\andspace
(y \ast z) \cdot \alpha(x)\\
= \alpha(x) \cdot (y \ast z).
\]
Therefore, the equality
\[
(y \cdot x) \ast \alpha(z) = (y \ast z) \cdot \alpha(x)
\]
holds for all $x,y,z \in A$ (which is equivalent to \eqref{rightmult'}) if and only if \eqref{rightmult''} holds for all $x,y,z \in A$.
\end{proof}

Let us note that every non-trivial commutative Hom-associative algebra has a canonical non-trivial Hom-Novikov-Poisson algebra structure.  To prove this result, we need the following preliminary observation, which will be used many times below.

%%%%%%%%%%%%%%%%%%%
\begin{lemma}
\label{lem:comm}
Let $(A,\cdot,\alpha)$ be a commutative Hom-associative algebra.  Then the expressions
\[
(x\cdot y)\cdot \alpha(z) = \alpha(x) \cdot (y \cdot z)
\]
are both invariant under every permutation of $x,y,z \in A$.
\end{lemma}
%%%%%%%%%%%%%%%%%%%

\begin{proof}
The expression $(x\cdot y) \cdot \alpha(z)$ is symmetric in $x$ and $y$ because $\cdot$ is symmetric.  Moreover, it is symmetric in $y$ and $z$ because
\begin{equation}
\label{xyz}
\begin{split}
(x \cdot y) \cdot \alpha(z)
&= (y \cdot x) \cdot \alpha(z) \quad\text{(by commutativity)}\\
&= \alpha(y) \cdot (x \cdot z) \quad\text{(by Hom-associativity)}\\
&= (x \cdot z) \cdot \alpha(y) \quad\text{(by commutativity)}.
\end{split}
\end{equation}
Since the symmetric group $S_3$ in three letters is generated by the transpositions $(1~2)$ and $(2~3)$, we conclude that the expression $(x\cdot y)\cdot\alpha(z)$ is invariant under permutations of $x,y,z$.
\end{proof}

%%%%%%%%%%%%%%%%%%%%%
\begin{proposition}
\label{prop:homass}
Let $(A,\cdot,\alpha)$ be a commutative Hom-associative algebra.  Then $(A,\cdot,\cdot,\alpha)$ is a Hom-Novikov-Poisson algebra.
\end{proposition}
%%%%%%%%%%%%%%%%%%%%%

\begin{proof}
Indeed, in this case the defining identities \eqref{homnovikov} of a Hom-Novikov algebra coincide with the compatibility conditions \eqref{hnp}.  The condition \eqref{leftsymmetric} holds because $as_A = 0$ by Hom-associativity.  The condition \eqref{rightmult} holds by Lemma \ref{lem:comm}.
\end{proof}

The next result says that Hom-Novikov-Poisson algebras are closed under twisting by self-weak morphisms.  As we will see, this property is unique to Hom-Novikov-Poisson algebras, as Novikov-Poisson algebras are \emph{not} closed under such twistings.

%%%%%%%%%%%%%%%%%
\begin{theorem}
\label{thm:twist}
Let $(A, \cdot, \ast, \alpha)$ be a Hom-Novikov-Poisson algebra, and let $\beta \colon A \to A$ be a weak morphism.  Then
\[
A_\beta = (A,\beta\cdot,\beta\ast,\beta\alpha)
\]
is also a Hom-Novikov-Poisson algebra.  Moreover, if $A$ is multiplicative and $\beta$ is a morphism, then $A_\beta$ is also multiplicative.
\end{theorem}
%%%%%%%%%%%%%%%%%

\begin{proof}
To see that $A_\beta$ is a Hom-Novikov-Poisson algebra, one applies $\beta^2$ to (i) the Hom-associativity \eqref{homassociative} of $(A,\cdot,\alpha)$, (ii) the conditions \eqref{homnovikov} of the Hom-Novikov algebra $(A,\ast,\alpha)$, and (iii) the compatibility conditions \eqref{hnp}.  The multiplicativity assertion follows from a direct computation.
\end{proof}

Let us discuss some special cases of Theorem \ref{thm:twist}.  The following special case says that every multiplicative Hom-Novikov-Poisson algebra induces a sequence of multiplicative Hom-Novikov-Poisson algebras by twisting against its own twisting map.

%%%%%%%%%%%%%%%%%%%
\begin{corollary}
\label{cor1:twist}
Let $(A, \cdot, \ast, \alpha)$ be a multiplicative Hom-Novikov-Poisson algebra.  Then
\[
A^n = (A, \alpha^n\cdot,  \alpha^n\ast, \alpha^{n+1})
\]
is also a multiplicative Hom-Novikov-Poisson algebra for each $n \geq 0$.
\end{corollary}
%%%%%%%%%%%%%%%%%%%

\begin{proof}
Since $A$ is multiplicative, $\alpha^n \colon A \to A$ is a morphism.  By Theorem \ref{thm:twist} $A_{\alpha^n} = A^n$ is a multiplicative Hom-Novikov-Poisson algebra
\end{proof}

The following result is the special case of Corollary \ref{cor1:twist} with $\cdot = 0$.

%%%%%%%%%%%%%%%%%%%%%%
\begin{corollary}
\label{cor1.5:twist}
Let $(A,*,\alpha)$ be a multiplicative Hom-Novikov algebra.  Then
\[
A^n = (A, \alpha^n\ast, \alpha^{n+1})
\]
is also a multiplicative Hom-Novikov algebra for each $n \geq 0$.
\end{corollary}
%%%%%%%%%%%%%%%%%%%%%%

The following result is the special case of Corollary \ref{cor1:twist} with $* = 0$.

%%%%%%%%%%%%%%%%%%%%%%
\begin{corollary}
\label{cor1.6:twist}
Let $(A,\cdot,\alpha)$ be a multiplicative commutative Hom-associative algebra.  Then
\[
A^n = (A, \alpha^n\cdot, \alpha^{n+1})
\]
is also a multiplicative commutative Hom-associative algebra for each $n \geq 0$.
\end{corollary}
%%%%%%%%%%%%%%%%%%%%%%

The following result is the $\alpha = Id_A$ special case of Theorem \ref{thm:twist}.

%%%%%%%%%%%%%%%%%%
\begin{corollary}
\label{cor2:twist}
Let $(A,\cdot,\ast)$ be a Novikov-Poisson algebra and $\beta \colon A \to A$ be a morphism.  Then
\[
A_\beta = (A,\beta\cdot,\beta\ast,\beta)
\]
is a multiplicative Hom-Novikov-Poisson algebra.
\end{corollary}
%%%%%%%%%%%%%%%%%%

Corollary \ref{cor2:twist} says that multiplicative Hom-Novikov-Poisson algebras can be constructed from Novikov-Poisson algebras and their morphisms.  A construction result of this form was first given by the author in \cite{yau2} for $G$-Hom-associative algebras.  This twisting construction highlights the fact that the category of Novikov-Poisson algebras is \emph{not} closed under twisting by self-morphisms.  In view of Theorem \ref{thm:twist}, this is a major conceptual difference between Hom-Novikov-Poisson algebras and Novikov-Poisson algebras.

The following special case of Corollary \ref{cor2:twist} is useful for constructing examples of Hom-Novikov-Poisson algebras.

%%%%%%%%%%%%%%%%%%
\begin{corollary}
\label{cor3:twist}
Let $(A,\mu)$ be a commutative associative algebra, $\partial \colon A \to A$ be a derivation, and $\alpha \colon A \to A$ be an algebra morphism such that $\alpha\partial = \partial\alpha$.  Then $A_\alpha = (A,\cdot, \ast, \alpha)$ is a multiplicative Hom-Novikov-Poisson algebra, where
\[
x \cdot y = \alpha\mu(x,y)\quad\text{and}\quad
x \ast y = \alpha\mu(x,\partial y)
\]
for $x,y \in A$.
\end{corollary}
%%%%%%%%%%%%%%%%%%

\begin{proof}
It is known that $(A,\mu,\bullet)$ is a Novikov-Poisson algebra \cite{xu1} (Lemma 2.1), where
\[
x \bullet y = \mu(x,\partial y).
\]
(That $(A,\bullet)$ is a Novikov algebra has been known since \cite{dg3}.)  The assumptions on $\alpha$ imply that $\alpha \bullet = \bullet \alpha^{\otimes 2}$, so $\alpha$ is a morphism of Novikov-Poisson algebras.  The result now follows from Corollary \ref{cor2:twist}.
\end{proof}

The following examples illustrate Corollary \ref{cor3:twist}.

%%%%%%%%%%%%%%%%%
\begin{example}
\label{ex:jdelta}
Starting with an example in \cite{xu2}, we construct infinite-dimensional multiplicative Hom-Novikov-Poisson algebras that are not Novikov-Poisson algebras.

Let the ground field be the field $\bC$ of complex numbers.  Let $J$ be either $\{0\}$ or $\bN = \{0,1,2,\ldots\}$, $\Delta$ be an additive subgroup of $\bC$, and $A$ be the $\bC$-module spanned by the symbols $\{x_a^j \colon a \in \Delta,\, j \in J\}$.  Then $A$ is a commutative associative algebra with the multiplication
\[
x_a^j \cdot x_b^k = x_{a+b}^{j+k}
\]
for $a,b \in \Delta$ and $j,k \in J$.  It has a multiplicative identity
\[
1 = x_0^0.
\]
The map $\partial \colon A \to A$ defined by
\[
\partial(x_a^j) = ax_a^j + jx_a^{j-1}
\]
for $a \in \Delta$ and $j \in J$ is a derivation on $(A,\cdot)$ by \cite{xu2} (2.18).  Therefore, by \cite{xu2} Lemma 2.1, $(A,\cdot,\ast)$ is a Novikov-Poisson algebra, where
\[
\begin{split}
x_a^j \ast x_b^k &= x_a^j \cdot \partial (x_b^k)\\
&= bx_{a+b}^{j+k} + kx_{a+b}^{j+k-1}
\end{split}
\]
for $a,b \in \Delta$ and $j,k \in J$.  Note that if either $J = \bN$ or $\Delta \not= \{0\}$, then $A$ is infinite-dimensional.

Let $\alpha \colon A \to A$ be the linear map determined by
\begin{equation}
\label{alphaxai}
\alpha(x_a^j) = e^{a}x_a^j
\end{equation}
for $a \in \Delta$ and $j \in J$.  It is straightforward to check that $\alpha \colon A \to A$ is an algebra morphism on $(A,\cdot)$ such that $\alpha\partial = \partial\alpha$.  By Corollary \ref{cor3:twist} there is a multiplicative Hom-Novikov-Poisson algebra
\begin{equation}
\label{aalpha1}
A_\alpha = (A,\cdot_\alpha,\ast_\alpha,\alpha),
\end{equation}
in which
\[
\begin{split}
x_a^j \cdot_\alpha x_b^k &= e^{a+b}x_{a+b}^{j+k},\\
x_a^j \ast_\alpha x_b^k &= e^{a+b}\left(bx_{a+b}^{j+k} + kx_{a+b}^{j+k-1}\right)
\end{split}
\]
for $a,b \in \Delta$ and $j,k \in J$.

Note that $(A,\cdot_\alpha,\ast_\alpha)$ is not a Novikov-Poisson algebra, provided $\Delta \nsubseteq \{2n\pi i \colon n \in \bZ\}$.  Indeed, suppose there exists $a \in \Delta$ such that $a \not= 2n\pi i$ for any integer $n$.  Then $e^a \not= 0,1$.  Now on the one hand we have
\[
(1 \cdot_\alpha 1) \cdot_\alpha x^0_a = e^a x^0_{a}.
\]
On the other hand, we have
\[
1 \cdot_\alpha (1 \cdot_\alpha x^0_a) = e^{2a} x^0_a,
\]
which shows that $\cdot_\alpha$ is not associative.  Hence $(A,\cdot_\alpha,\ast_\alpha)$ is not a Novikov-Poisson algebra.
\qed
\end{example}
%%%%%%%%%%%%%%%%%

%%%%%%%%%%%%%%%%%
\begin{example}
\label{ex:nilpotent}
Let $(A,\cdot)$ be a commutative associative algebra and $\partial \colon A \to A$ be a nilpotent derivation, i.e., $\partial$ is a derivation such that $\partial^n = 0$ for some $n \geq 2$.  One can check that the formal exponential map
\[
\varphi = \sum_{k=0}^{n-1} \, \frac{\partial^k}{k!} = Id + \partial + \frac{\partial^2}{2!} + \cdots + \frac{\partial^{n-1}}{(n-1)!}
\]
is a well-defined algebra automorphism on $A$ \cite{abe} (p.26).  Moreover, $\varphi$ commutes with $\partial$ because $\varphi$ is a polynomial in $\partial$.  By Corollary \ref{cor3:twist} there is a multiplicative Hom-Novikov-Poisson algebra
\begin{equation}
\label{avarphi}
A_\varphi = (A,\cdot_\varphi,\ast_\varphi,\varphi),
\end{equation}
in which
\[
\begin{split}
f \cdot_\varphi g &= \varphi(f) \cdot \varphi(g) = \sum_{m=0}^{2n-2} \left(\sum_{k+j=m}\, \frac{(\partial^k f) \cdot (\partial^j g)}{k!j!}\right),\\
f \ast_\varphi g &= \varphi(f) \cdot \varphi(\partial g) = \sum_{m=0}^{2n-3} \left(\sum_{k+j=m}\, \frac{(\partial^k f) \cdot (\partial^{j+1} g)}{k!j!}\right)
\end{split}
\]
for $f,g \in A$.
\qed
\end{example}
%%%%%%%%%%%%%%%%%

%%%%%%%%%%%%%%%%%
\begin{example}
This example is a special case of Example \ref{ex:nilpotent} in which $(A,\cdot_\varphi,\ast_\varphi)$ is not a Novikov-Poisson algebra.

Let $A$ be the truncated polynomial algebra $\bk[x]/(x^N)$ for some integer $N \geq 2$.  The differential operator $\partial = \frac{d}{dx}$ is a nilpotent derivation on $A$ with $\partial^N = 0$.  The associated formal exponential is
\[
\varphi = \sum_{k=0}^{N-1} \frac{1}{k!}\left(\frac{d}{dx}\right)^k.
\]
As in Example \ref{ex:nilpotent}, $A_\varphi$ in \eqref{avarphi} is a multiplicative Hom-Novikov-Poisson algebra.  We claim that $(A,\cdot_\varphi,\ast_\varphi)$ is not a Novikov-Poisson algebra.  It suffices to show that $\cdot_\varphi$ is not associative.  The element $x \in A$ satisfies
\[
\varphi^k(x) = x + k
\]
for all $k \geq 1$.  Now on the one hand, we have
\[
\begin{split}
(x \cdot_\varphi x) \cdot_\varphi \varphi(x)
&= (x+1)^2 \cdot_\varphi (x+1)\\
&= (x+2)^3.
\end{split}
\]
On the other hand, we have
\[
\begin{split}
x \cdot_\varphi (x \cdot_\varphi \varphi(x))
&= x \cdot_\varphi ((x+1)(x+2))\\
&= (x+1)(x+2)(x+3).
\end{split}
\]
We have shown that $\cdot_\varphi$ is not associative, so $(A,\cdot_\varphi,\ast_\varphi)$ is not a Novikov-Poisson algebra.
\qed
\end{example}
%%%%%%%%%%%%%%%%%

%%%%%%%%%%%%%%%%%
\begin{example}
Let $A$ be the polynomial algebra $\bk[x_1, \ldots , x_n]$.  Fix an integer $j \in \{1,\ldots,n\}$.  The partial differential operator $\partial = \frac{\partial}{\partial x_j}$ with respect to $x_j$ is a derivation on the commutative associative algebra $A$.

For $1 \leq i \leq n$ let $c_i \in \bk$ be arbitrary scalars, and let $\alpha \colon A \to A$ be the algebra morphism determined by
\[
\alpha(x_i) = x_i + c_i.
\]
Then $\alpha$ commutes with $\partial$ because
\[
\begin{split}
\partial(\alpha(x_i^k))
&= \partial((x_i+c_i)^k)\\
&= \delta_{i,j} k(x_i+c_i)^{k-1}\\
&= \alpha(\delta_{i,j} kx_i^{k-1})\\
&= \alpha(\partial(x_i^k)).
\end{split}
\]
By Corollary \ref{cor3:twist} there is a multiplicative Hom-Novikov-Poisson algebra $(A,\cdot,\ast,\alpha)$, in which
\[
\begin{split}
f \cdot g &= f(x_1+c_1, \ldots, x_n+c_n)g(x_1+c_1, \ldots, x_n+c_n),\\
f \ast g &= f(x_1+c_1, \ldots, x_n+c_n)\frac{\partial g}{\partial x_j}(x_1+c_1, \ldots, x_n+c_n)
\end{split}
\]
for $f = f(x_1,\ldots,x_n)$ and $g = g(x_1,\ldots,x_n)$ in $A$.

Note that, as long as $\alpha$ is not the identity map, $(A,\cdot,\ast)$ is not a Novikov-Poisson algebra because $\cdot$ is not associative.  Indeed, suppose $c_i \not= 0$.  Then we have
\[
(1 \cdot 1) \cdot x_i = x_i+c_i,
\]
whereas
\[
1 \cdot (1 \cdot x_i) = x_i+2c_i.
\]
This shows that $\cdot$ is not associative, so $(A,\cdot,\ast)$ is not a Novikov-Poisson algebra.
\qed
\end{example}
%%%%%%%%%%%%%%%%%

%%%%%%%%%%%%%%%%%%%%%%%%%%
\section{Tensor products}
\label{sec:tensor}
%%%%%%%%%%%%%%%%%%%%%%%%%%

The tensor product of two Novikov algebras is usually not a Novikov algebra in a non-trivial way.  One reason Novikov-Poisson algebras were introduced by Xu \cite{xu1} was that, unlike Novikov algebras, Novikov-Poisson algebras are closed under tensor products non-trivially (\cite{xu1} Theorem 4.1).  Here we show that the much larger class of Hom-Novikov-Poisson algebras is also closed under tensor products.

%%%%%%%%%%%%%%%%%%
\begin{theorem}
\label{thm:tensor}
Let $(A_i,\cdot_i,\ast_i,\alpha_i)$ be Hom-Novikov-Poisson algebras for $i = 1,2$, and let $A = A_1 \otimes A_2$.  Define the operations $\alpha \colon A \to A$ and $\cdot, \ast \colon A^{\otimes 2} \to A$ by:
\[
\begin{split}
\alpha &= \alpha_1 \otimes \alpha_2,\\
(\xtwo) \cdot (\ytwo) &= (x_1 \cdot_1 y_1) \otimes (x_2 \cdot_2 y_2),\\
(\xtwo) \ast (\ytwo) &= (x_1 \ast_1 y_1) \otimes (x_2 \cdot_2 y_2) + (x_1 \cdot_1 y_1) \otimes (x_2 \ast_2 y_2)
\end{split}
\]
for $x_i,y_i \in A_i$.  Then $(A,\cdot,\ast,\alpha)$ is a Hom-Novikov-Poisson algebra.  Moreover, if both $A_i$ are multiplicative, then $A$ is also multiplicative.
\end{theorem}
%%%%%%%%%%%%%%%%%%

\begin{proof}
It is easy to see that $(A,\cdot,\alpha)$ is a commutative Hom-associative algebra and that $A$ is multiplicative if both $A_i$ are.  We show that $(A,\ast,\alpha)$ is a Hom-Novikov algebra in Lemma \ref{lem1:tensor} below.  The compatibility conditions \eqref{hnp} for $A$ are proved in Lemma \ref{lem2:tensor} below.
\end{proof}

%%%%%%%%%%%%%%%%%%%%%
\begin{lemma}
\label{lem1:tensor}
Under the assumptions of Theorem \ref{thm:tensor}, $(A,\ast,\alpha)$ is a Hom-Novikov algebra.
\end{lemma}
%%%%%%%%%%%%%%%%%%%%%

\begin{proof}
To improve readability, we omit the subscripts in $\ast_i$ and $\alpha_i$ and write $x \cdot_i y$ as $xy$.  Let us first prove \eqref{rightmult}.  Pick $x = \xtwo$, $y = \ytwo$, and $z = \ztwo$ in $A$.  We must show that $(x \ast y) \ast \alpha(z)$ is symmetric in $y$ and $z$.  We have:
\begin{equation}
\label{xyzast}
\begin{split}
(x \ast y) \ast \alpha(z)
&= \{x_1\ast y_1 \otimes x_2y_2 + x_1y_1 \otimes x_2\ast y_2\} \ast (\alpha(z_1) \otimes \alpha(z_2))\\
&= \underbrace{(x_1\ast y_1)\ast \alpha(z_1) \otimes (x_2y_2)\alpha(z_2)}_{a}
+ \underbrace{(x_1\ast y_1)\alpha(z_1) \otimes (x_2y_2)\ast \alpha(z_2)}_{b(x,y,z)}\\
&\relphantom{} + \underbrace{(x_1y_1)\ast \alpha(z_1) \otimes (x_2\ast y_2)\alpha(z_2)}_{c(x,y,z)}
+ \underbrace{(x_1y_1)\alpha(z_1) \otimes (x_2\ast y_2)\ast \alpha(z_2)}_{d}.
\end{split}
\end{equation}
The terms $a$ and $d$ are symmetric in $y$ and $z$ by \eqref{rightmult} in both $A_i$ and Lemma \ref{lem:comm}.  Moreover, we have
\[
\begin{split}
b(x,y,z) &= (x_1 z_1)\ast \alpha(y_1) \otimes (x_2 \ast z_2)\alpha(y_2)\quad\text{(by \eqref{rightmult2})}\\
&= c(x,z,y).
\end{split}
\]
This shows that $(x \ast y) \ast \alpha(z)$ is symmetric in $y$ and $z$, so \eqref{rightmult} holds in $A$.

For \eqref{leftsymmetric} in $A$, we must show that the Hom-associator
\[
as_*(x,y,z) = (x \ast y) \ast \alpha(z) - \alpha(x) \ast (y \ast z)
\]
with respect to $*$ in $A$ is symmetric in $x$ and $y$.  Let us compute the second term in the Hom-associator:
\begin{equation}
\label{xyzast2}
\begin{split}
\alpha(x) \ast (y \ast z)
&= (\alpha(x_1) \otimes \alpha(x_2))\ast \{y_1 \ast z_1 \otimes y_2z_2 + y_1z_1 \otimes y_2\ast z_2\}\\
&= \underbrace{\alpha(x_1) \ast (y_1 \ast z_1) \otimes \alpha(x_2)(y_2z_2)}_{a'}
+ \underbrace{\alpha(x_1)(y_1 \ast z_1) \otimes \alpha(x_2) \ast (y_2z_2)}_{b'}\\
&\relphantom{} + \underbrace{\alpha(x_1) \ast (y_1z_1) \otimes \alpha(x_2)(y_2\ast z_2)}_{c'}
+ \underbrace{\alpha(x_1)(y_1z_1) \otimes \alpha(x_2) \ast(y_2 \ast z_2)}_{d'}
\end{split}
\end{equation}
Using the notations in \eqref{xyzast} (with $b = b(x,y,z)$ and $c = c(x,y,z)$), Hom-associativity, and Lemma \ref{lem:rightmult}, we have:
\begin{equation}
\label{aa}
\begin{split}
a - a' &= as_{*_1}(x_1,y_1,z_1) \otimes (x_2y_2)\alpha(z_2),\\
d - d' &= (x_1y_1)\alpha(z_1) \otimes as_{*_2}(x_2,y_2,z_2),\\
c - b' &= (x_1y_1) \ast \alpha(z_1) \otimes as_{A_2}(x_2,y_2,z_2),\\
b - c' &= as_{A_1}(x_1,y_1,z_1) \otimes (x_2y_2) \ast \alpha(z_2).
\end{split}
\end{equation}
Here $as_{*_i}$ is the Hom-associator with respect to $\ast_i$, and $as_{A_i}$ is the mixed Hom-associator in $A_i$ (Definition \ref{def2:homassociator}).  By the commutativity of $\cdot_i$, \eqref{leftsymmetric}, and \eqref{mixedass} in $A_i$, it follows that $(a-a')$, $(d-d')$, $(c-b')$, and $(b-c')$ are all symmetric in $x$ and $y$.  Therefore, we conclude from \eqref{xyzast}, \eqref{xyzast2}, and \eqref{aa} that $as_*$ in $A$ is also symmetric in $x$ and $y$, thereby proving \eqref{leftsymmetric} in $A$.
\end{proof}

%%%%%%%%%%%%%%%%%%%%%
\begin{lemma}
\label{lem2:tensor}
Under the assumptions of Theorem \ref{thm:tensor},  $(A,\cdot,\ast,\alpha)$ satisfies the compatibility conditions \eqref{hnp}.
\end{lemma}
%%%%%%%%%%%%%%%%%%%%%

\begin{proof}
Let us first prove \eqref{rightmult2} in $A$.  Using Lemma \ref{lem:comm}, the notations in Lemma \ref{lem1:tensor}, and \eqref{rightmult2} in $A_i$, we have:
\begin{equation}
\label{xyz2}
\begin{split}
(x * y) \cdot \alpha(z)
&= (x_1 * y_1 \otimes x_2y_2 + x_1y_1 \otimes x_2 * y_2) \cdot (\alpha(z_1) \otimes \alpha(z_2))\\
&= \underbrace{(x_1*y_1)\alpha(z_1) \otimes (x_2y_2)\alpha(z_2)}_{u}
+ \underbrace{(x_1y_1)\alpha(z_1) \otimes (x_2*y_2)\alpha(z_2)}_{v}\\
&= (x_1z_1)*\alpha(y_1) \otimes (x_2z_2)\alpha(y_2) + (x_1z_1)\alpha(y_1) \otimes (x_2z_2)*\alpha(y_2)\\
&= (x_1z_1 \otimes x_2z_2) * (\alpha(y_1) \otimes \alpha(y_2))\\
&= (x \cdot z) * \alpha(y).
\end{split}
\end{equation}
This proves \eqref{rightmult2} in $A$.  The labels $u$ and $v$ will be used in the next paragraph.

For \eqref{mixedass} in $A$, we must show that the mixed Hom-associator
\[
as_A(x,y,z) = (x*y) \cdot \alpha(z) - \alpha(x) * (y \cdot z)
\]
is symmetric in $x$ and $y$.  Let us compute the second term in the mixed Hom-associator:
\begin{equation}
\label{xyz3}
\begin{split}
\alpha(x) * (y \cdot z)
&= (\alpha(x_1) \otimes \alpha(x_2)) * (y_1z_1 \otimes y_2z_2)\\
&= \underbrace{\alpha(x_1)*(y_1z_1) \otimes \alpha(x_2)(y_2z_2)}_{u'}
+ \underbrace{\alpha(x_1)(y_1z_1) \otimes \alpha(x_2)*(y_2z_2)}_{v'}.
\end{split}
\end{equation}
Using Hom-associativity, we have:
\[
\begin{split}
u - u' &= as_{A_1}(x_1,y_1,z_1) \otimes (x_2y_2)\alpha(z_2),\\
v - v' &= (x_1y_1)\alpha(z_1) \otimes as_{A_2}(x_2,y_2,z_2).
\end{split}
\]
It follows from the commutativity of $\cdot_i$ and \eqref{mixedass} in $A_i$ that both $(u-u')$ and $(v-v')$ are symmetric in $x$ and $y$.  Therefore, the mixed Hom-associator
\[
as_A(x,y,z) = (u-u') + (v-v')
\]
is also symmetric in $x$ and $y$, thereby proving \eqref{mixedass} in $A$.
\end{proof}

Taking $\alpha_i = Id_{A_i}$ in Theorem \ref{thm:tensor}, we recover the following result, which is Theorem 4.1 in \cite{xu1}.

%%%%%%%%%%%%%%%%%%
\begin{corollary}
\label{cor1:tensor}
Let $(A_i,\cdot_i,\ast_i)$ be Novikov-Poisson algebras for $i = 1,2$, and let $A = A_1 \otimes A_2$.  Define the operations $\cdot, \ast \colon A^{\otimes 2} \to A$ by:
\[
\begin{split}
(\xtwo) \cdot (\ytwo) &= (x_1 \cdot_1 y_1) \otimes (x_2 \cdot_2 y_2),\\
(\xtwo) \ast (\ytwo) &= (x_1 \ast_1 y_1) \otimes (x_2 \cdot_2 y_2) + (x_1 \cdot_1 y_1) \otimes (x_2 \ast_2 y_2)
\end{split}
\]
for $x_i,y_i \in A_i$.  Then $(A,\cdot,\ast)$ is a Novikov-Poisson algebra.
\end{corollary}
%%%%%%%%%%%%%%%%%%

Theorem \ref{thm:tensor} can be used with the results and examples in the previous section to construct new Hom-Novikov-Poisson algebras. For example, the next result is obtained by first using Corollary \ref{cor1:twist} and then Theorem \ref{thm:tensor}.

%%%%%%%%%%%%%%%%%%%%%
\begin{corollary}
\label{cor2:tensor}
Let $(A_i,\cdot_i,\ast_i,\alpha_i)$ be multiplicative Hom-Novikov-Poisson algebras for $i = 1,2$, and let $A = A_1 \otimes A_2$.  For integers $n,m \geq 0$, define the operations $\alpha \colon A \to A$ and $\cdot, \ast \colon A^{\otimes 2} \to A$ by:
\[
\begin{split}
\alpha &= \alpha_1^{n+1} \otimes \alpha_2^{m+1},\\
(\xtwo) \cdot (\ytwo) &= \alpha_1^{n}(x_1 \cdot_1 y_1) \otimes \alpha_2^{m}(x_2 \cdot_2 y_2),\\
(\xtwo) \ast (\ytwo) &= \alpha_1^{n}(x_1 \ast_1 y_1) \otimes \alpha_2^{m}(x_2 \cdot_2 y_2) + \alpha_1^{n}(x_1 \cdot_1 y_1) \otimes \alpha_2^{m}(x_2 \ast_2 y_2)
\end{split}
\]
for $x_i,y_i \in A_i$.  Then $(A,\cdot,\ast,\alpha)$ is a multiplicative Hom-Novikov-Poisson algebra.
\end{corollary}
%%%%%%%%%%%%%%%%%%%%%

\begin{proof}
Indeed, using the notations in Corollary \ref{cor1:twist}, we have
\[
(A,\cdot,\ast,\alpha) = A_1^n \otimes A_2^m,
\]
where the tensor product is equipped with the Hom-Novikov-Poisson algebra structure in Theorem \ref{thm:tensor}.
\end{proof}

Observe that the tensor product in Theorem \ref{thm:tensor} and the operation $(-)^n$ in Corollary \ref{cor1:twist} commute.  In other words, given any two multiplicative Hom-Novikov-Poisson algebras $A_1$ and $A_2$, we have
\[
(A_1 \otimes A_2)^n = A_1^n \otimes A_2^n
\]
for all $n \geq 0$.

The following result is obtained by using Corollary \ref{cor2:twist} and then Theorem \ref{thm:tensor}.

%%%%%%%%%%%%%%%%%%%
\begin{corollary}
\label{cor3:tensor}
Let $(A_i,\cdot_i,*_i)$ be Novikov-Poisson algebras and $\beta_i \colon A_i \to A_i$ be morphisms for $i=1,2$.  Let $A = A_1 \otimes A_2$.  Define the operations $\beta \colon A \to A$ and $\cdot, \ast \colon A^{\otimes 2} \to A$ by:
\[
\begin{split}
\beta &= \beta_1 \otimes \beta_2,\\
(\xtwo) \cdot (\ytwo) &= \beta_1(x_1 \cdot_1 y_1) \otimes \beta_2(x_2 \cdot_2 y_2),\\
(\xtwo) \ast (\ytwo) &= \beta_1(x_1 \ast_1 y_1) \otimes \beta_2(x_2 \cdot_2 y_2) + \beta_1(x_1 \cdot_1 y_1) \otimes \beta_2(x_2 \ast_2 y_2)
\end{split}
\]
for $x_i,y_i \in A_i$.  Then $(A,\cdot,\ast,\alpha)$ is a multiplicative Hom-Novikov-Poisson algebra.
\end{corollary}
%%%%%%%%%%%%%%%%%%%

\begin{proof}
Indeed, in the notations of Corollary \ref{cor2:twist}, we have
\[
(A,\cdot,\ast,\alpha) = (A_1)_{\beta_1} \otimes (A_2)_{\beta_2},
\]
where the tensor product is equipped with the Hom-Novikov-Poisson algebra structure in Theorem \ref{thm:tensor}.
\end{proof}

Observe that the tensor product in Theorem \ref{thm:tensor} and the operation $(-)_\beta$ in Corollary \ref{cor2:twist} commute.  In other words, given any two Novikov-Poisson algebras $A_i$ and morphisms $\beta_i \colon A_i \to A_i$ for $i=1,2$, we have
\[
(A_1 \otimes A_2)_{\beta_1\otimes\beta_2} = (A_1)_{\beta_1} \otimes (A_2)_{\beta_2},
\]
where both tensor products are equipped with the Hom-Novikov-Poisson algebra structures in Theorem \ref{thm:tensor}.

%%%%%%%%%%%%%%%%%%%%%%%%%%%%%%%%%%%%%%%%%%%%%%%%%%%%%%%%
\section{Perturbations of Hom-Novikov-Poisson algebras}
\label{sec:perturb}
%%%%%%%%%%%%%%%%%%%%%%%%%%%%%%%%%%%%%%%%%%%%%%%%%%%%%%%%

The purpose of this section is to show that certain perturbations preserve Hom-Novikov-Poisson algebra structures. We need the following preliminary observation about perturbing the structure maps in a commutative Hom-associative algebra.

%%%%%%%%%%%%%%%%%%%%
\begin{lemma}
\label{lem1:perturb}
Let $(A,\cdot,\alpha)$ be a commutative Hom-associative algebra and $a$ be an element in $A$.  Define the operation $\diamond \colon A^{\otimes 2} \to A$ by
\begin{equation}
\label{diamondop}
x \diamond y = a \cdot (x \cdot y)
\end{equation}
for $x,y \in A$.  Then
\[
A' = (A,\diamond,\alpha^2)
\]
is also a commutative Hom-associative algebra.  Moreover, if $A$ is multiplicative and $\alpha^2(a) = a$, then $A'$ is also multiplicative.
\end{lemma}
%%%%%%%%%%%%%%%%%%%%

\begin{proof}
As usual we abbreviate $x \cdot y$ to $xy$.  The commutativity of $\diamond$ follows from that of $\cdot$.  The multiplicativity assertion is straightforward to check.  To show that $A'$ is Hom-associative, pick $x,y,z \in A$.  Then we have:
\[
\begin{split}
(x \diamond y) \diamond \alpha^2(z)
&= a\{(a(xy))\alpha^2(z)\}\\
&= a\{\alpha(a)((xy)\alpha(z))\}\byhomass\\
&= a\{\alpha(a)(\alpha(x)(yz))\}\byhomass\\
&= a\{\alpha^2(x)(a(yz))\}\quad\text{(by Lemma \ref{lem:comm})}\\
&= \alpha^2(x) \diamond (y \diamond z).
\end{split}
\]
This shows that $A'$ is Hom-associative.
\end{proof}

The following result says that the structure maps of a multiplicative Hom-Novikov-Poisson algebra can be perturbed by a suitable element and its own twisting map.

%%%%%%%%%%%%%%%%%%%%
\begin{theorem}
\label{thm1:perturb}
Let $(A,\cdot,*,\alpha)$ be a multiplicative Hom-Novikov-Poisson algebra and $a \in A$ be an element such that $\alpha^2(a) = a$.  Then
\[
A' = (A,\diamond,*_\alpha,\alpha^2)
\]
is also a multiplicative Hom-Novikov-Poisson algebra, where
\[
\begin{split}
x \diamond y &= a \cdot (x \cdot y),\\
x *_\alpha y &= \alpha(x * y) = \alpha(x) * \alpha(y)
\end{split}
\]
for all $x,y \in A$.
\end{theorem}
%%%%%%%%%%%%%%%%%%%%

\begin{proof}
By Lemma \ref{lem1:perturb} we know that $(A,\diamond,\alpha^2)$ is a multiplicative commutative Hom-associative algebra.  Also, $(A,*_\alpha,\alpha^2)$ is a multiplicative Hom-Novikov algebra by Corollary \ref{cor1.5:twist} (the $n=1$ case).  It remains to prove the compatibility conditions \eqref{hnp} for $A'$.

To prove the compatibility condition \eqref{rightmult2} (or equivalently \eqref{rightmult''}) for $A'$, pick $x,y,z \in A$.  With $x \cdot y$ written as $xy$, we have:
\[
\begin{split}
(x \diamond y) *_\alpha \alpha^2(z)
&= \{\alpha(a)\alpha(xy)\} * \alpha^3(z)\bymult\\
&= \alpha^2(a)\{\alpha(xy) * \alpha^2(z)\quad\text{(by \eqref{rightmult''} in $A$)}\\
&= a\{(\alpha(x)\alpha(y)) * \alpha^2(z)\}\bymult\\
&= a\{\alpha^2(x)(\alpha(y) * \alpha(z))\}\quad\text{(by \eqref{rightmult''} in $A$)}\\
&= \alpha^2(x) \diamond (y *_\alpha z).
\end{split}
\]
This proves that $A'$ satisfies \eqref{rightmult''}, which is equivalent to \eqref{rightmult2} by Lemma \ref{lem:rightmult}.

To prove the compatibility condition \eqref{mixedass} for $A'$, we must show that the mixed Hom-associator (Definition \ref{def2:homassociator})
\begin{equation}
\label{mha0}
as_{A'}(x,y,z) = (x *_\alpha y) \diamond \alpha^2(z) - \alpha^2(x) *_\alpha (y \diamond z)
\end{equation}
in $A'$ is symmetric in $x$ and $y$.  The first term in the mixed Hom-associator is:
\begin{equation}
\label{mha1}
\begin{split}
(x *_\alpha y) \diamond \alpha^2(z)
&= a \{(\alpha(x) * \alpha(y))\alpha^2(z)\}\\
&= \{(\alpha(x) * \alpha(y))\alpha^2(z)\}\alpha^2(a)\bycomm\\
&= \{\alpha^2(x) * \alpha^2(y)\} (\alpha^2(z)\alpha(a))\byhomass\\
&= \{\alpha^2(x) * \alpha^2(y)\} (\alpha(\alpha(z)a)).
\end{split}
\end{equation}
The second term in the mixed Hom-associator is:
\begin{equation}
\label{mha2}
\begin{split}
\alpha^2(x) *_\alpha (y \diamond z)
&= \alpha^3(x) * \alpha(a(yz))\\
&= \alpha^3(x) * \{\alpha(a)(\alpha(y)\alpha(z))\}\\
&= \alpha^3(x) * \{\alpha^2(y)(\alpha(z)a)\}\quad\text{(by Lemma \ref{lem:comm})}.
\end{split}
\end{equation}
Using \eqref{mha0}, \eqref{mha1}, and \eqref{mha2}, it follows that the mixed Hom-associators of $A'$ and $A$ are related as follows:
\[
\begin{split}
as_{A'}(x,y,z)
&= \{\alpha^2(x) * \alpha^2(y)\} (\alpha(\alpha(z)a)) - \alpha^3(x) * \{\alpha^2(y)(\alpha(z)a)\}\\
&= as_A(\alpha^2(x),\alpha^2(y),\alpha(z)a)
\end{split}
\]
Since $as_A$ is symmetric in its first two arguments \eqref{mixedass}, we conclude that $as_{A'}(x,y,z)$ is symmetric in $x$ and $y$.
\end{proof}

Setting $\alpha = Id_A$ in Theorem \ref{thm1:perturb}, we recover Lemma 2.4 in \cite{xu2}:

%%%%%%%%%%%%%%%%%%%%%
\begin{corollary}
\label{cor:perturb}
Let $(A,\cdot,*)$ be a Novikov-Poisson algebra and $a \in A$ be an arbitrary element.  Then $(A,\diamond,*)$ is also a Novikov-Poisson algebra, where
\[
x \diamond y = a \cdot x \cdot y
\]
for all $x,y \in A$.
\end{corollary}
%%%%%%%%%%%%%%%%%%%%%

The next result is a variation on the theme of Theorem \ref{thm1:perturb}.  It gives a way to perturb a Hom-Novikov-Poisson algebra structure using a suitable element and its own twisting map.

%%%%%%%%%%%%%%%%%%%%
\begin{theorem}
\label{thm2:perturb}
Let $(A,\cdot,*,\alpha)$ be a multiplicative Hom-Novikov-Poisson algebra and $a \in A$ be an element such that $\alpha^2(a) = a$.  Then
\[
\Abar = (A,\dotalpha,\times,\alpha^2)
\]
is also a multiplicative Hom-Novikov-Poisson algebra, where
\[
\begin{split}
x \dotalpha y &= \alpha(x \cdot y) = \alpha(x) \cdot \alpha(y),\\
x \times y &= \alpha(x) * \alpha(y) + a \cdot (x \cdot y)
\end{split}
\]
for all $x,y \in A$.
\end{theorem}
%%%%%%%%%%%%%%%%%%%%

\begin{proof}
Throughout the proof of Theorem \ref{thm2:perturb}, we abbreviate $x \cdot y$ to $xy$.  By Corollary \ref{cor1.6:twist} (the $n=1$ case) $(A,\dotalpha,\alpha^2)$ is a multiplicative commutative Hom-associative algebra.  We show that $(A,\times,\alpha^2)$ is a multiplicative Hom-Novikov algebra in Lemma \ref{lem2a:perturb} below. The compatibility conditions \eqref{hnp} for $\Abar$ are proved in Lemma \ref{lem2b:perturb} below.
\end{proof}

%%%%%%%%%%%%%%%%%%%%
\begin{lemma}
\label{lem2a:perturb}
Under the assumptions of Theorem \ref{thm2:perturb}, $(A,\times,\alpha^2)$ is a multiplicative Hom-Novikov algebra.
\end{lemma}
%%%%%%%%%%%%%%%%%%%%

\begin{proof}
The multiplicativity of $(A,\times,\alpha^2)$ follows from that of $A$ and the assumption $\alpha^2(a) = a$.  Pick arbitrary elements $x,y,z \in A$.  To check the condition \eqref{rightmult}, we must show that the expression $(x \times y) \times \alpha^2(z)$ is symmetric in $y$ and $z$.  Expanding this expression in terms of $\cdot$, $*$, and $\alpha$, we have:
\begin{equation}
\label{xyztimes}
\begin{split}
(x \times y) \times \alpha^2(z)
&= (\alpha(x) * \alpha(y) + a(xy)) \times \alpha^2(z)\\
&= \underbrace{(\alpha^2(x) * \alpha^2(y)) * \alpha^3(z)}_{t} + \underbrace{\alpha(a(xy)) * \alpha^3(z)}_{u(x,y,z)}\\
&\relphantom{} + \underbrace{a\{(\alpha(x)*\alpha(y)) \alpha^2(z)\}}_{v(x,y,z)}
+ \underbrace{a\{(a(xy))\alpha^2(z)\}}_{w}
\end{split}
\end{equation}
The term $t$ in \eqref{xyztimes} is symmetric in $y$ and $z$ by \eqref{rightmult} in $(A,*,\alpha)$.  By the Hom-associativity in $(A,\cdot,\alpha)$, the term $w$ in \eqref{xyztimes} can be rewritten as
\begin{equation}
\label{w}
w = a\{\alpha(a)((xy)\alpha(z))\},
\end{equation}
which is symmetric in $x$, $y$, and $z$ by Lemma \ref{lem:comm}.  Next we claim that
\[
u(x,y,z) = v(x,z,y)
\]
in \eqref{xyztimes}.  To prove this, we compute as follows:
\[
\begin{split}
u(x,y,z) &= \{\alpha(a)(\alpha(x)\alpha(y))\} * \alpha^3(z)\bymult\\
&= \{(a\alpha(y))\alpha^2(x)\} * \alpha^3(z)\quad\text{(by Lemma \ref{lem:comm})}\\
&= \alpha(a\alpha(y))\{\alpha^2(x) * \alpha^2(z)\}\quad\text{(by Lemma \ref{lem:rightmult})}\\
&= \{\alpha(a)\alpha^2(y)\} \alpha(\alpha(x) * \alpha(z))\bymult\\
&= \alpha^2(a)\{\alpha^2(y) (\alpha(x) * \alpha(z))\}\byhomass\\
&= a\{(\alpha(x) * \alpha(z)) \alpha^2(y)\}\\
&= v(x,z,y).
\end{split}
\]
By \eqref{xyztimes} it follows that $(x \times y) \times \alpha^2(z)$ is symmetric in $y$ and $z$, thereby proving \eqref{rightmult} for $(A,\times,\alpha^2)$.

For \eqref{leftsymmetric} we must show that the Hom-associator
\[
as_\times (x,y,z) = (x \times y) \times \alpha^2(z) - \alpha^2(x) \times (y \times z)
\]
in $(A,\times,\alpha^2)$ is symmetric in $x$ and $y$.  Expanding the second term in this Hom-associator, we have:
\begin{equation}
\label{xyztimes'}
\begin{split}
\alpha^2(x) \times (y \times z)
&= \alpha^2(x) \times (\alpha(y) * \alpha(z) + a(yz))\\
&= \underbrace{\alpha^3(x) * (\alpha^2(y) * \alpha^2(z))}_{t'} + \underbrace{\alpha^3(x) * \alpha(a(yz))}_{u'(x,y,z)}\\
&\relphantom{} + \underbrace{a\{\alpha^2(x)(\alpha(y) * \alpha(z))\}}_{v'}
+ \underbrace{a\{\alpha^2(x)(a(yz))\}}_{w'}.
\end{split}
\end{equation}
Using the notations in \eqref{xyztimes} (with $u = u(x,y,z)$) and \eqref{xyztimes'}, we have
\begin{equation}
\label{homasstimes}
as_\times (x,y,z) = (t-t') + (w-w') + (u-v') + (v(x,y,z) - u'(x,y,z)).
\end{equation}
We now show that
\[
w-w' = 0 = u - v'
\]
and that both $(t-t')$ and $(v(x,y,z) - u'(x,y,z))$ are symmetric in $x$ and $y$.

The first summand on the right-hand side of \eqref{homasstimes} is:
\[
t - t' = as_*(\alpha^2(x),\alpha^2(y),\alpha^2(z)),
\]
where $as_*$ is the Hom-associator of $A$ with respect to $*$. Since $(A,*,\alpha)$ is a Hom-Novikov algebra, $as_*$ is symmetric in its first two variables by \eqref{leftsymmetric}, which implies that $(t-t')$ is symmetric in $x$ and $y$.

For the second summand on the right-hand side of \eqref{homasstimes}, note that by \eqref{w} we have:
\[
\begin{split}
w &= a\{\alpha(a)((xy)\alpha(z))\}\\
&= a\{\alpha(a)((yz)\alpha(x))\}\quad\text{(by Lemma \ref{lem:comm})}\\
&= a\{\alpha^2(x)(a(yz))\}\quad\text{(by Lemma \ref{lem:comm})}\\
&= w'.
\end{split}
\]
Therefore, we have $w-w' = 0$.

For the third summand on the right-hand side of \eqref{homasstimes}, we have:
\[
\begin{split}
u &= \{\alpha(a)(\alpha(x)\alpha(y))\} * \alpha^3(z)\\
&= \{\alpha^2(y)(a\alpha(x))\} * \alpha^3(z)\quad\text{(by Lemma \ref{lem:comm})}\\
&= (\alpha^2(y) * \alpha^2(z)) \alpha(a\alpha(x))\quad\text{(by \eqref{rightmult2})}\\
&= \{\alpha(a)\alpha^2(x)\}\alpha(\alpha(y) * \alpha(z))\bymult\\
&= \alpha^2(a)\{\alpha^2(x)(\alpha(y) * \alpha(z))\}\byhomass\\
&= v'.
\end{split}
\]
Therefore, we have $u-v' = 0$.

For the last summand on the right-hand side of \eqref{homasstimes}, first note that:
\begin{equation}
\label{v}
\begin{split}
v(x,y,z)  &= \alpha^2(a)\{(\alpha(x) * \alpha(y)) \alpha^2(z)\}\\
&= \{\alpha(a)\alpha^2(z)\} \alpha(\alpha(x) * \alpha(y))\quad\text{(by Lemma \ref{lem:comm})}\\
&= \alpha(a\alpha(z)) \{\alpha^2(x) * \alpha^2(y)\}\bymult\\
&= \{(a\alpha(z))\alpha^2(x)\} * \alpha^3(y)\quad\text{(by Lemma \ref{lem:rightmult})}\\
&= \{(a\alpha(z)) * \alpha^2(y)\} \alpha^3(x)\quad\text{(by \eqref{rightmult2})}.
\end{split}
\end{equation}
On the other hand, we have:
\[
\begin{split}
u'(x,y,z) &= \alpha^3(x) * \{\alpha(a)(\alpha(y)\alpha(z))\}\\
&= \alpha^3(x) * \{(a\alpha(z))\alpha^2(y)\}\quad\text{(by Lemma \ref{lem:comm})}\\
&= \{\alpha^2(x) * (a\alpha(z))\}\alpha^3(y) + \alpha(a\alpha(z)) * (\alpha^2(x)\alpha^2(y))\\
&\relphantom{} - \{(a\alpha(z)) * \alpha^2(x)\}\alpha^3(y)\quad\text{(by left-symmetry of $as_A(\alpha^2(x),a\alpha(z),\alpha^2(y))$ \eqref{mixedass})}\\
&= \{\alpha^2(x)\alpha^2(y)\} * \alpha(a\alpha(z)) + \alpha(a\alpha(z)) * (\alpha^2(x)\alpha^2(y))\\
&\relphantom{} - v(y,x,z)\quad\text{(by \eqref{rightmult2} and \eqref{v})}.
\end{split}
\]
Therefore, we have:
\[
\begin{split}
v(x,y,z) - u'(x,y,z)
&= v(x,y,z) + v(y,x,z)\\
&\relphantom{} - \{\alpha^2(x)\alpha^2(y)\} * \alpha(a\alpha(z)) - \alpha(a\alpha(z)) * (\alpha^2(x)\alpha^2(y)),
\end{split}
\]
which is symmetric in $x$ and $y$ because $\cdot$ is commutative.

In summary, the previous four paragraphs and \eqref{homasstimes} show that the Hom-associator in $(A,\times,\alpha^2)$ is symmetric in $x$ and $y$, thereby proving \eqref{leftsymmetric}.
\end{proof}

%%%%%%%%%%%%%%%%%%%%
\begin{lemma}
\label{lem2b:perturb}
Under the assumptions of Theorem \ref{thm2:perturb}, $\Abar$ satisfies the compatibility conditions \eqref{hnp}.
\end{lemma}
%%%%%%%%%%%%%%%%%%%%

\begin{proof}
Pick elements $x,y,z \in A$.  To prove the compatibility condition \eqref{rightmult2}, or equivalently \eqref{rightmult''}, for $\Abar$, we compute as follows:
\[
\begin{split}
&(x \dotalpha y) \times \alpha^2(z)
= (\alpha(x)\alpha(y)) \times \alpha^2(z)\\
&= \{\alpha^2(x)\alpha^2(y)\} * \alpha^3(z) + a\{(\alpha(x)\alpha(y))\alpha^2(z)\}\\
&= \alpha^3(x)\{\alpha^2(y) * \alpha^2(z)\} + \alpha^2(a)\{\alpha^2(x)(\alpha(y)\alpha(z))\}\quad\text{(by \eqref{rightmult''} and Hom-associativity)}\\
&= \alpha^3(x)\{\alpha^2(y) * \alpha^2(z)\} + \alpha^3(x)\{\alpha(a)(\alpha(y)\alpha(z))\}\quad\text{(by Lemma \ref{lem:comm})}\\
&= \alpha^3(x) \alpha\{\alpha(y) * \alpha(z) + a(yz)\}\bymult\\
&= \alpha^2(x) \dotalpha (y \times z).
\end{split}
\]
To prove the compatibility condition \eqref{mixedass} for $\Abar$, we must show that the mixed Hom-associator
\[
as_{\Abar}(x,y,z) = (x \times y) \dotalpha \alpha^2(z) - \alpha^2(x) \times (y \dotalpha z)
\]
of $\Abar$ is symmetric in $x$ and $y$.  The first summand in this mixed Hom-associator is:
\begin{equation}
\label{abar1}
\begin{split}
(x \times y) \dotalpha \alpha^2(z)
&= \alpha\{\alpha(x) * \alpha(y) + a(xy)\} \alpha^3(z)\\
&= \{\alpha^2(x) * \alpha^2(y)\}\alpha^3(z) + \{\alpha(a)(\alpha(x)\alpha(y))\}\alpha^3(z).
\end{split}
\end{equation}
The second summand in the mixed Hom-associator is:
\begin{equation}
\label{abar2}
\begin{split}
\alpha^2(x) \times (y \dotalpha z)
&= \alpha^2(x) \times (\alpha(y)\alpha(z))\\
&= \alpha^3(x) * (\alpha^2(y)\alpha^2(z)) + a\{\alpha^2(x)(\alpha(y)\alpha(z))\}\\
&= \alpha^3(x) * (\alpha^2(y)\alpha^2(z)) + \alpha^2(a)\{(\alpha(x)\alpha(y))\alpha^2(z)\}\byhomass\\
&= \alpha^3(x) * (\alpha^2(y)\alpha^2(z)) + \{\alpha(a)(\alpha(x)\alpha(y))\}\alpha^3(z)\byhomass.
\end{split}
\end{equation}
Combining \eqref{abar1} and \eqref{abar2}, it follows that the mixed Hom-associators of $\Abar$ and $A$ are related as follows:
\[
\begin{split}
as_{\Abar}(x,y,z)
&= \{\alpha^2(x) * \alpha^2(y)\}\alpha^3(z) - \alpha^3(x) * (\alpha^2(y)\alpha^2(z))\\
&= as_A(\alpha^2(x),\alpha^2(y),\alpha^2(z)).
\end{split}
\]
Since $as_A$ is left-symmetric by \eqref{mixedass}, we conclude that $as_{\Abar}$ is symmetric in $x$ and $y$, as desired.
\end{proof}

With Lemmas \ref{lem2a:perturb} and \ref{lem2b:perturb} proved, the proof of Theorem \ref{thm2:perturb} is complete.  We now discuss some special cases of Theorem \ref{thm2:perturb}.

Setting $\alpha = Id_A$ in Theorem \ref{thm2:perturb}, we recover Lemma 2.3 in \cite{xu2}:

%%%%%%%%%%%%%%%%%%%%%
\begin{corollary}
\label{cor2:perturb}
Let $(A,\cdot,*)$ be a Novikov-Poisson algebra and $a \in A$ be an arbitrary element.  Then $(A,\cdot,\times)$ is also a Novikov-Poisson algebra, where
\[
x \times y = x * y + a \cdot x \cdot y
\]
for all $x,y \in A$.
\end{corollary}
%%%%%%%%%%%%%%%%%%%%%

Forgetting about the Hom-associative product $\dotalpha$ in Theorem \ref{thm2:perturb}, we obtain the following result, which gives a non-trivial way to construct a Hom-Novikov algebra from a Hom-Novikov-Poisson algebra.

%%%%%%%%%%%%%%%%%%%%%%%
\begin{corollary}
\label{cor4.1:perturb}
Let $(A,\cdot,*,\alpha)$ be a multiplicative Hom-Novikov-Poisson algebra and $a \in A$ be an element such that $\alpha^2(a) = a$.  Then $(A,\times,\alpha^2)$ is a multiplicative Hom-Novikov algebra, where
\[
x \times y = \alpha(x)*\alpha(y) + a \cdot (x \cdot y)
\]
for all $x,y \in A$.
\end{corollary}
%%%%%%%%%%%%%%%%%%%%%%%

Setting $\alpha = Id_A$ in Corollary \ref{cor4.1:perturb}, we obtain the following special case of Corollary \ref{cor2:perturb}.

%%%%%%%%%%%%%%%%%%%%%%%
\begin{corollary}
\label{cor4.2:perturb}
Let $(A,\cdot,*)$ be a Novikov-Poisson algebra and $a \in A$ be an arbitrary element.  Then $(A,\times)$ is a Novikov algebra, where
\[
x \times y = x*y + a \cdot x \cdot y
\]
for all $x,y \in A$.
\end{corollary}
%%%%%%%%%%%%%%%%%%%%%%%

The following perturbation result is obtained by combining Theorems \ref{thm1:perturb} and \ref{thm2:perturb}.

%%%%%%%%%%%%%%%%%%%%%
\begin{corollary}
\label{cor3:perturb}
Let $(A,\cdot,*,\alpha)$ be a multiplicative Hom-Novikov-Poisson algebra and $a,b \in A$ be elements such that $\alpha^2(a) = a$ and $\alpha^4(b) = b$.  Then
\[
\Atilde = (A,\diamondalpha,\boxtimes,\alpha^4)
\]
is also a multiplicative Hom-Novikov-Poisson algebra, where
\[
\begin{split}
x \diamondalpha y &= \alpha(b) \cdot \alpha^2(x \cdot y),\\
x \boxtimes y &= \alpha^3(x*y) + a \cdot \alpha^2(x \cdot y)
\end{split}
\]
for all $x,y \in A$.
\end{corollary}
%%%%%%%%%%%%%%%%%%%%%

\begin{proof}
By Theorem \ref{thm2:perturb}, $\Abar = (A,\dotalpha,\times,\alpha^2)$ is a multiplicative Hom-Novikov-Poisson algebra.  Now apply Theorem \ref{thm1:perturb} to $\Abar$ and the element $b \in A$, which satisfies $(\alpha^2)^2(b) = b$.  We obtain a multiplicative Hom-Novikov-Poisson algebra $(\Abar)'$, which is $\Atilde$ above.
\end{proof}

Setting $\alpha = Id_A$ in Corollary \ref{cor3:perturb}, we recover Theorem 2.5 in \cite{xu2}:

%%%%%%%%%%%%%%%%%%%%
\begin{corollary}
\label{cor4:perturb}
Let $(A,\cdot,*)$ be a Novikov-Poisson algebra and $a,b \in A$ be arbitrary elements.  Then $(A,\diamond,\boxtimes)$ is also a Novikov-Poisson algebra, where
\[
\begin{split}
x \diamond y &= b \cdot x \cdot y,\\
x \boxtimes y &= x*y + a \cdot x \cdot y
\end{split}
\]
for all $x,y \in A$.
\end{corollary}
%%%%%%%%%%%%%%%%%%%%

The following result is another special case of Corollary \ref{cor3:perturb}.

%%%%%%%%%%%%%%%%%%%%
\begin{corollary}
\label{cor5:perturb}
Let $A$ be a commutative associative algebra, $\partial \colon A \to A$ be a derivation, $\alpha \colon A \to A$ be an algebra morphism such that $\alpha\partial = \partial\alpha$, and $a,b \in A$ be elements such that $\alpha^2(a) = a$ and $\alpha^4(b) = b$.  Then $(A,\cdot,*,\alpha^4)$ is a multiplicative Hom-Novikov-Poisson algebra, where
\[
\begin{split}
x \cdot y &= \alpha^2(b)\alpha^4(xy),\\
x * y &= \alpha^4(x\partial(y)) + \alpha(a)\alpha^4(xy)
\end{split}
\]
for all $x,y \in A$.
\end{corollary}
%%%%%%%%%%%%%%%%%%%%

\begin{proof}
By Corollary \ref{cor3:twist},
\[
A_\alpha = (A,\alpha\mu,\alpha\mu(Id\otimes\partial),\alpha)
\]
is a multiplicative Hom-Novikov-Poisson algebra, where $\mu$ is the given commutative associative product in $A$.  Now apply Corollary \ref{cor3:perturb} to $A_\alpha$ and the elements $a$ and $b$.  The result is a multiplicative Hom-Novikov-Poisson algebra $\widetilde{A_\alpha}$, whose operations are as stated above.
\end{proof}

Setting $\alpha = Id_A$ in Corollary \ref{cor5:perturb}, we recover Corollary 2.6 in \cite{xu2} (see also \cite{filippov}):

%%%%%%%%%%%%%%%%%%%%%
\begin{corollary}
\label{cor6:perturb}
Let $A$ be a commutative associative algebra, $\partial \colon A \to A$ be a derivation, and $a,b \in A$ be arbitrary elements.  Then $(A,\cdot,*)$ is a Novikov-Poisson algebra, where
\[
\begin{split}
x \cdot y &= bxy,\\
x * y &= x\partial(y) + axy
\end{split}
\]
for all $x,y \in A$.
\end{corollary}
%%%%%%%%%%%%%%%%%%%%%

%%%%%%%%%%%%%%%%%%%%%%%%%%%%%%%%%%%%%%%%%%%%%%%%%%%%%%%%%%
\section{From Hom-Novikov-Poisson algebras to Hom-Poisson algebras}
\label{sec:poisson}
%%%%%%%%%%%%%%%%%%%%%%%%%%%%%%%%%%%%%%%%%%%%%%%%%%%%%%%%%%

The purpose of this section is to show how Hom-Poisson algebras arise from Hom-Novikov-Poisson algebras.  A Poisson algebra is a commutative associative algebra with a Lie algebra structure that satisfies the Leibniz identity.  To define a Hom-Poisson algebra, let us first recall the relevant definitions.

%%%%%%%%%%%%%%%%%%%
\begin{definition}
\label{def:homlie}
A \textbf{Hom-Lie algebra} \cite{hls,ms,yau} is a Hom-algebra $(A,[,],\alpha)$ such that $[,]$ is anti-symmetric and that the \textbf{Hom-Jacobi identity}
\[
[[x,y],\alpha(z)] + [[z,x],\alpha(y)] + [[y,z],\alpha(x)] = 0
\]
holds for all $x,y,z \in A$.
\end{definition}
%%%%%%%%%%%%%%%%%%%

%%%%%%%%%%%%%%%%%%%
\begin{definition}
\label{def:hompoisson}
A \textbf{Hom-Poisson algebra} \cite{ms3} $(A,\cdot,[,],\alpha)$ consists of
\begin{enumerate}
\item
a commutative Hom-associative algebra $(A,\cdot,\alpha)$ and
\item
a Hom-Lie algebra $(A,[,],\alpha)$
\end{enumerate}
such that the \textbf{Hom-Leibniz identity}
\begin{equation}
\label{homleibniz}
[\alpha(x), y \cdot z] = [x,y]\cdot \alpha(z) + \alpha(y)\cdot[x,z]
\end{equation}
holds for all $x,y,z \in A$.
\end{definition}
%%%%%%%%%%%%%%%%%%%

Hom-Poisson algebras were defined \cite{ms3} as Hom-type generalizations of Poisson algebras.  In the context of deformations, Hom-Poisson algebras were shown in \cite{ms3} to be related to commutative Hom-associative algebras as Poisson algebras are related to commutative associative algebras.  Further properties of (non-commutative) Hom-Poisson algebras can be found in \cite{yau10}.

Both a Hom-Novikov-Poisson algebra and a Hom-Poisson algebra have an underlying commutative Hom-associative algebra.  So it makes sense to ask whether a Hom-Poisson algebra can be constructed from a Hom-Novikov-Poisson algebra by taking the commutator bracket of the Hom-Novikov product.  To answer this question, we make the following definitions.

%%%%%%%%%%%%%%%%%%%%
\begin{definition}
Let $(A,\cdot,*,\alpha)$ be a double Hom-algebra.  Its \textbf{left Hom-associator} $as_A^l \colon A^{\otimes 3} \to A$ is defined as
\[
as_A^l = *(\cdot \otimes \alpha - \alpha \otimes \cdot),
\]
or equivalently
\[
as_A^l(x,y,z) = (x \cdot y) * \alpha(z) - \alpha(x) * (y \cdot z)
\]
for $x,y,z \in A$.  The double Hom-algebra $A$ is called \textbf{left Hom-associative} if $as_A^l = 0$.
\end{definition}
%%%%%%%%%%%%%%%%%%%%

Using Lemma \ref{lem:rightmult}, the left Hom-associator $as_A^l$ in a Hom-Novikov-Poisson algebra $A$ is equivalent to
\begin{equation}
\label{lefthomass}
as_A^l(x,y,z) = \alpha(x) \cdot (y * z) - \alpha(x) * (y \cdot z)
\end{equation}
for all $x,y,z \in A$.

%%%%%%%%%%%%%%%%%%%%
\begin{definition}
\label{def:hpa}
Let $(A,\cdot,*,\alpha)$ be Hom-Novikov-Poisson algebra.  Then $A$ is called \textbf{admissible} if the double Hom-algebra
\begin{equation}
\label{aminus}
A^- = (A,\cdot,[,],\alpha)
\end{equation}
is a Hom-Poisson algebra, where
\[
[x,y] = x*y - y*x
\]
for all $x,y \in A$.
\end{definition}
%%%%%%%%%%%%%%%%%%%%

The following result gives a necessary and sufficient condition under which a Hom-Novikov-Poisson algebra is admissible.  It is the Hom-type generalization of an observation in \cite{zbm}.

%%%%%%%%%%%%%%%%%
\begin{theorem}
\label{thm:adm}
Let $(A,\cdot,*,\alpha)$ be a Hom-Novikov-Poisson algebra.  Then $A$ is admissible if and only if it is left Hom-associative.
\end{theorem}
%%%%%%%%%%%%%%%%%

\begin{proof}
By definition $(A,\cdot,\alpha)$ is a commutative Hom-associative algebra.  Moreover, $(A,[,],\alpha)$ is a Hom-Lie algebra by Proposition 4.3 in \cite{ms}.  (Equivalently, one can expand the left-hand side of the Hom-Jacobi identity in terms of $*$ and observe that \eqref{leftsymmetric} implies that the resulting sum is $0$.)  Therefore, we must show that $A^-$ satisfies the Hom-Leibniz identity \eqref{homleibniz} if and only if $as_A^l = 0$.

Pick elements $x,y,z \in A$.  As usual we abbreviate $x \cdot y$ to $xy$.  The left-hand side of the Hom-Leibniz identity \eqref{homleibniz} for $A^-$ is:
\begin{equation}
\label{hl1}
\begin{split}
& [\alpha(x), yz]\\
&= \alpha(x) * (yz) - (yz) * \alpha(x)\\
&= (x*y)\alpha(z) - (y*x)\alpha(z) + \underbrace{\alpha(y)*(xz)}_{p} - \alpha(y)(z * x) \quad\text{(by \eqref{mixed} and Lemma \ref{lem:rightmult})}
\end{split}
\end{equation}
The right-hand side of the Hom-Leibniz identity \eqref{homleibniz} for $A^-$ is:
\begin{equation}
\label{hl2}
\begin{split}
& [x,y]\alpha(z) + \alpha(y)[x,z]\\
&= (x*y)\alpha(z) - (y*x)\alpha(z) + \alpha(y)(x*z) - \alpha(y)(z*x)\\
&= (x*y)\alpha(z) - (y*x)\alpha(z) + \underbrace{(yx)*\alpha(z)}_{q} - \alpha(y)(z*x)\quad\text{(by Lemma \ref{lem:rightmult})}\\
\end{split}
\end{equation}
It follows from \eqref{hl1} and \eqref{hl2} that $A^-$ satisfies the Hom-Leibniz identity if and only if
\[
\begin{split}
0 &= q-p\\
&= (yx)*\alpha(z) - \alpha(y)*(xz)\\
&= as_A^l(y,x,z).
\end{split}
\]
Since $x,y,z \in A$ are arbitrary, we conclude that $A^-$ is a Hom-Poisson algebra if and only if $A$ is left Hom-associative.
\end{proof}

%%%%%%%%%%%%%%%%%%%%%%
\begin{example}
\label{ex:derivation}
Consider the multiplicative Hom-Novikov-Poisson algebra $A_\alpha = (A,\cdot, \ast, \alpha)$ in Corollary \ref{cor3:twist}.  Here $A$ is a commutative associative algebra, $\partial \colon A \to A$ is a derivation, and $\alpha \colon A \to A$ is an algebra morphism such that $\alpha\partial = \partial\alpha$.  The operations $\cdot$ and $*$ are
\[
x \cdot y = \alpha(xy)\quad\text{and}\quad
x \ast y = \alpha(x\partial y)
\]
for $x,y \in A$.  Then $A_\alpha$ is admissible if and only if
\begin{equation}
\label{xydz}
\alpha^2(xy\partial z) = 0
\end{equation}
for all $x,y,z \in A$.  Indeed, $A_\alpha$ is left Hom-associative if and only if
\[
\begin{split}
0 &= \alpha(x) * (y \cdot z) - \alpha(x) \cdot (y * z)\\
&= \alpha\{\alpha(x)\partial(\alpha(yz))\} - \alpha\{\alpha(x)\alpha(y\partial z)\}\\
&= \alpha^2(x\partial(yz) - xy\partial z)\\
&= \alpha^2(xz\partial y).
\end{split}
\]
Therefore, $A_\alpha$ satisfies \eqref{xydz} if and only if it is left Hom-associative, which by Theorem \ref{thm:adm} is equivalent to $A_\alpha$ being admissible.
\qed
\end{example}
%%%%%%%%%%%%%%%%%%%%%%

In the rest of this section, we show that admissibility is compatible with the constructions in the previous sections.  We begin with the twisting constructions in section \ref{sec:hnp}.

%%%%%%%%%%%%%%%%%%%%
\begin{corollary}
\label{cor1:adm}
Let $(A,\cdot,*,\alpha)$ be an admissible Hom-Novikov-Poisson algebra and $\beta \colon A \to A$ be a weak morphism.  Then
\[
A_\beta = (A,\beta\cdot,\beta\ast,\beta\alpha)
\]
is also an admissible Hom-Novikov-Poisson algebra.  Moreover, if $A$ is multiplicative and $\beta$ is a morphism, then $A_\beta$ is also multiplicative.
\end{corollary}
%%%%%%%%%%%%%%%%%%%%

\begin{proof}
By Theorem \ref{thm:twist} $A_\beta$ is a Hom-Novikov-Poisson algebra, and if $A$ is multiplicative and $\beta$ is a morphism, then $A_\beta$ is multiplicative.  The left Hom-associators in $A$ and $A_\beta$ are related as
\[
as_{A_\beta}^l = \beta^2 as_A^l.
\]
Since $A$ is left Hom-associative by Theorem \ref{thm:adm}, it follows that so is $A_\beta$.  Therefore, by Theorem \ref{thm:adm} again $A_\beta$ is admissible.
\end{proof}

In the context of Corollary \ref{cor1:adm}, the Hom-Lie bracket in the Hom-Poisson algebra $A_\beta^-$ is given by
\[
\beta(x * y) - \beta(y * x) = \beta[x,y],
\]
where $[,]$ is the Hom-Lie bracket in the Hom-Poisson algebra $A^-$.

The next result is a special case of Corollary \ref{cor1:adm}.

%%%%%%%%%%%%%%%%%%
\begin{corollary}
\label{cor2:adm}
Let $(A,\cdot,*,\alpha)$ be a multiplicative admissible Hom-Novikov-Poisson algebra.  Then so is
\[
A^n = (A, \alpha^n\cdot,  \alpha^n\ast, \alpha^{n+1})
\]
for each $n \geq 0$.
\end{corollary}
%%%%%%%%%%%%%%%%%%

\begin{proof}
The multiplicativity of $A$ implies that $\alpha^n$ is a morphism.  Now apply Corollary \ref{cor1:adm} with $\beta = \alpha^n$.
\end{proof}

Next we observe that admissibility is preserved by tensor products.

%%%%%%%%%%%%%%%%%%
\begin{corollary}
\label{cor3:adm}
Let $(A_i,\cdot_i,\ast_i,\alpha_i)$ be admissible Hom-Novikov-Poisson algebras for $i = 1,2$, and let $A = A_1 \otimes A_2$ be the Hom-Novikov-Poisson algebra in Theorem \ref{thm:tensor}.  Then $A$ is admissible.
\end{corollary}
%%%%%%%%%%%%%%%%%%

\begin{proof}
By Theorem \ref{thm:adm} we need to show that $A$ is left Hom-associative.   Pick $x = \xtwo$, $y = \ytwo$, and $z = \ztwo$ in $A$.  Then we have:
\[
\begin{split}
&\alpha(x) \cdot (y * z)\\
&= \alpha(x_1) \cdot (y_1*z_1) \otimes \alpha(x_2) \cdot (y_2 \cdot z_2)
+ \alpha(x_1)\cdot(y_1\cdot z_1) \otimes \alpha(x_2)\cdot(y_2*z_2)\\
&= \alpha(x_1) * (y_1\cdot z_1) \otimes \alpha(x_2)\cdot (y_2\cdot z_2)
+ \alpha(x_1)\cdot (y_1\cdot z_1) \otimes \alpha(x_2) * (y_2\cdot z_2) \quad\text{(by $as^l_{A_i} = 0$)}\\
&= \alpha(x) * (y\cdot z).
\end{split}
\]
Therefore, $A$ is left Hom-associative by \eqref{lefthomass}.
\end{proof}

In the context of Corollary \ref{cor3:adm}, the Hom-Lie bracket in the Hom-Poisson algebra $A^-$ is given by:
\[
\begin{split}
[\xtwo,\ytwo] &= (\xtwo) * (\ytwo) - (\ytwo) * (\xtwo)\\
&= x_1*y_1 \otimes x_2y_2 + x_1y_1 \otimes x_2*y_2\\
&\relphantom{} - y_1*x_1 \otimes y_2x_2 - y_1x_1 \otimes y_2*x_2\\
&= [x_1,y_1] \otimes x_2y_2 + x_1y_1 \otimes [x_2,y_2].
\end{split}
\]
The last equality holds because $\cdot_i$ is commutative, and $[x_i,y_i]$ is the Hom-Lie bracket in the Hom-Poisson algebra $A_i^-$.

Next we observe that admissibility is preserved by the perturbations in Theorem \ref{thm1:perturb}.

%%%%%%%%%%%%%%%%%%
\begin{corollary}
\label{cor4:adm}
Let $(A,\cdot,*,\alpha)$ be a multiplicative admissible Hom-Novikov-Poisson algebra and $a \in A$ be an element such that $\alpha^2(a) = a$.  Then the multiplicative Hom-Novikov-Poisson algebra
\[
A' = (A,\diamond,*_\alpha,\alpha^2)
\]
in Theorem \ref{thm1:perturb} is also admissible.
\end{corollary}
%%%%%%%%%%%%%%%%%%

\begin{proof}
By Theorem \ref{thm:adm} we need to show that $A'$ is left Hom-associative.  Pick $x,y,z \in A$.  Recall that
\[
x \diamond y = a \cdot (x \cdot y) \andspace
x *_\alpha y = \alpha(x * y) = \alpha(x) * \alpha(y)
\]
for $x,y \in A$.  The first summand in the left Hom-associator \eqref{lefthomass} in $A'$ is:
\[
\begin{split}
\alpha^2(x) \diamond (y *_\alpha z)
&= a\{\alpha^2(x)(\alpha(y) * \alpha(z))\}\\
&= \alpha^2(a)\{\alpha^2(x)(\alpha(y) * \alpha(z))\}\\
&= \alpha^3(x)\{\alpha(a)(\alpha(y) * \alpha(z))\} \quad\text{(by Lemma \ref{lem:comm})}\\
&= \alpha^3(x)\{(a\alpha(y)) * \alpha^2(z)\} \quad\text{(by Lemma \ref{lem:rightmult})}.
\end{split}
\]
The second summand in the left Hom-associator \eqref{lefthomass} in $A'$ is:
\[
\begin{split}
\alpha^2(x) *_\alpha (y \diamond z)
&= \alpha^3(x) * \{\alpha(a)(\alpha(y)\alpha(z))\}\\
&= \alpha^3(x) * \{(a\alpha(y))\alpha^2(z)\}\byhomass.
\end{split}
\]
Therefore, the left Hom-associators \eqref{lefthomass} in $A'$ and $A$ are related as:
\[
\begin{split}
as_{A'}^l(x,y,z)
&=  \alpha^3(x)\{(a\alpha(y)) * \alpha^2(z)\} - \alpha^3(x) * \{(a\alpha(y))\alpha^2(z)\}\\
&= as_A^l(\alpha^2(x), a\alpha(y), \alpha^2(z)).
\end{split}
\]
Since $A$ is left Hom-associative by Theorem \ref{thm:adm}, it follows that so is $A'$.
\end{proof}

In the context of Corollary \ref{cor4:adm}, the Hom-Lie bracket in the Hom-Poisson algebra $(A')^-$ is given by
\[
x*_\alpha y - y*_\alpha x = \alpha(x*y - y*x) = \alpha[x,y],
\]
where $[,]$ is the Hom-Lie bracket in the Hom-Poisson algebra $A^-$.

%%==============%%
%%              %%
%%  References  %%
%%              %%
%%==============%%

\end{document}